\documentclass[12pt]{amsart}
\usepackage[english]{babel}
\usepackage{amsmath}
\usepackage{amsthm}
\usepackage{amssymb}
\usepackage{mathrsfs}
\usepackage{enumerate}
\usepackage[notcite, final, notref]{showkeys}
\usepackage{dsfont}

\usepackage[dvips]{color}
\newtheorem{theorem}{Theorem}[section]

\newtheorem{lemma}[theorem]{Lemma}
\newtheorem{proposition}[theorem]{Proposition}

\newtheorem{definition}[theorem]{Definition}

\theoremstyle{definition}
\newtheorem{remark}[theorem]{Remark}

\newcommand{\e}{\epsilon}

\usepackage{xcolor}

\title[]{Adaptative decomposition: the case of the Drury-Arveson space}

\author[D. Alpay]{Daniel Alpay}
\address{(DA) Department of Mathematics\\
Ben-Gurion University of the Negev\\
Beer-Sheva 84105 Israel}
\email{dany@math.bgu.ac.il}

\author[F. Colombo]{Fabrizio Colombo}
\address{(FC) Politecnico di
Milano\\Dipartimento di Matematica\\Via E. Bonardi, 9\\20133 Milano,
Italy}
\email{fabrizio.colombo@polimi.it}

\author[T. Qian]{Tao Qian}
\address{(TQ) Department of Mathematics\\
Universityof Macau\\
Macau}
\email{fsttq@umac.mo}

\author[I. Sabadini]{Irene Sabadini}
\address{(IS) Politecnico di
Milano\\Dipartimento di Matematica\\Via E. Bonardi, 9\\20133 Milano\\Italy}
\email{irene.sabadini@polimi.it}

\thanks{
The authors thank the
Macao Government FDCT 098/2012/A3 and D. Alpay thanks the Earl Katz family for endowing the chair
which supported his research.
}

\begin{document}
\maketitle
\begin{abstract}
The maximum selection principle allows to give expansions, in an adaptive way, of functions in the Hardy space $\mathbf H_2$ of the disk
in terms of Blaschke products. The expansion is specific to the given function. Blaschke factors and products have counterparts in
the unit ball of $\mathbb C^N$, and this fact allows us to extend in the present paper
the maximum selection principle to the case of functions in
the Drury-Arveson space of functions analytic in the unit ball of $\mathbb C^N$. This will give rise to an algorithm  which is a
variation in this higher dimensional case of the greedy algorithm. We also introduce infinite Blaschke products in this setting
and study their convergence.

 \end{abstract}

\date{today}
\tableofcontents
\section{Introduction}
\setcounter{equation}{0}
In \cite{QWa1} the authors introduced an algorithm based on the maximum selection principle, to decompose a given function of the Hardy space $\mathbf H_2(\mathbb D)$ of the unit disk
into intrinsic components which correspond to modified Blaschke products
\begin{equation}\label{TMsystem}
B_n(z)= \frac{\sqrt{1-|a_n|^2}}{1-z\overline{a_n}} \prod_{k=1}^{n-1} \frac{z-a_k}{1-z\overline{a_k}}, \qquad n=1,2,\ldots
\end{equation}
where the points $a_n\in\mathbb D$ are adaptively chosen according  to the given function.
These points $a_n$ do not necessarily satisfy the so-called hyperbolic non-separability condition
\begin{equation}\label{conditionTM}
\sum_{n=1}^\infty 1-|a_n|=\infty,
\end{equation}
and so the functions $B_n(z)$ do not necessarily form a complete system in $\mathbf H_2(\mathbb D)$.
This decomposition may be obtained in an adaptive way, see \cite{tao2015}, making
the algorithm more efficient than the greedy algorithm of which it is a variation.\\

In \cite{acqs} the above algorithm is extended to the matrix-valued case and the choice of a point and of a projection is based at each step on the
maximal selection principle. The extension is possible because of the existence of matrix-valued Blaschke factors and is based on the
existence of solutions of interpolation problems in the matrix-valued Hardy space of the disk.\\

When leaving the realm of one complex variable, a number of possibilities occur, and in particular the unit ball $\mathbb B_N$
of $\mathbb C^N$ and the
polydisk. The polydisk case will be studied in a future publication. In this paper we focus on the case of the unit ball. For the present purposes, it is more convenient to consider the
Drury-Arveson space rather than the Hardy space of the ball, and we extend some of the results of \cite{QWa1} and \cite{acqs} to
the setting of the Drury-Arveson space, denoted here ${\mathbf H}({\mathbb B}_N)$.
This is the space with reproducing kernel
\[
\frac{1}{1-\langle z,w\rangle},\quad z,w\in\mathbb B_N,
\]
with
\[
\langle z,w\rangle=\sum_{u=1}^Nz_u\overline{w_u}=zw^*,
\]
where $z=(z_1,\ldots, z_N)$ and $w=(w_1,\ldots, w_N)$ belong to $\mathbb B_N$. This space has a long history (see for instance
\cite{arveson-acta,agler-carthy,MR1882259,btv,MR80c:47010,zbMATH06526232}) and is used in the proof of a von Neumann inequality for row contractions.
Interpolation inside the space ${\mathbf H}({\mathbb B}_N)$ was done in \cite{akap1}. A key tool in \cite{akap1} was the existence in the ball of the counterpart of a Blaschke factor (appearing in
\cite{rudin-ball}; see \eqref{blasBN} below). The existence of these Blaschke factors and the fact that one can solve interpolation
problems in $\mathbf H({\mathbb B}_N)$ allow us to develop the asserted extension.\\

The approach in \cite{akap1} is based on the solution of Gleason's problem. For completeness we recall that
given a space, say $\mathcal F$, of functions analytic in $\Omega\subset\mathbb C^N$, Gleason's problem consists in finding
for every $f\in\mathcal F$ and $a=(a_1,\ldots,a_N)\in\Omega$, functions $g_1(z,a),\ldots,g_N(z,a)\in\mathcal F$ and such that
\begin{equation}
\label{gleasonnnne}
f(z)-f(a)=\sum_{u=1}^N(z_u-a_u)g_u(z,a),\quad z\in\Omega.
\end{equation}

Using power series, one sees that there always exist analytic functions satisfying  \eqref{gleasonnnne}. The requirement is that one can choose them in $\mathcal F$.\\

The paper consists of five sections besides the introduction. In Sections 2 and 3 we review some basic facts on the Drury-Arveson space,
and on the interpolation in it. The latter will be necessary to prove the maximum selection principle. This principle is proved in Section 4.
In Section 5 we prove the convergence of the algorithm. In the last section, which is of independent interest,
we consider infinite Blaschke products. When $N>1$ the $a_n$ in \eqref{conditionTM} are vectors in $\mathbb B_N$ and condition
\eqref{conditionTM} is replaced by the requirement
\[
\sum_{n=1}^\infty\sqrt{1-a_na_n^*}<\infty.
\]

We note that most of the analysis presented here still holds for general complete Nevanlinna-Pick kernels, that is kernels of the form
\[
\frac{1}{c(z)\overline{c(w)}-\langle d(z), d(w)\rangle_{\mathcal H}},
\]
where $c$ is scalar and $d$ is $\mathcal H$-valued where $\mathcal H$ is some Hilbert space or more generally, in some reproducing kernel Hilbert spaces in which Gleason's problem is solvable with bounded operators; see \cite{ak4} for the latter.\\

\section{The Drury-Arveson space}
\setcounter{equation}{0}
We use the multi-index notations
\[
z^\alpha=z_1^{\alpha_1}\cdots z_N^{\alpha_N},\quad\text{\rm and}\quad \alpha!=\alpha_1!\cdots \alpha_N!,
\]
with $z=(z_1,\ldots,z_N)\in\mathbb C^N\quad\text{and}\quad \alpha=(\alpha_1,\ldots,\alpha_N)\in\mathbb N_0^N$.
For $z,w\in\mathbb B_N$ we have
\begin{equation}
\frac{1}{1-zw^*}=\sum_{\alpha\in\mathbb N_0^N}\frac{|\alpha|!}{\alpha!}z^\alpha{\overline{w}}^\alpha.
\label{drury!!!}
\end{equation}
The function \eqref{drury!!!} is thus positive definite in $\mathbb B_N$. The associated reproducing kernel Hilbert space, which we denote by
$\mathbf H({\mathbb B}_N)$, is called the Drury-Arveson space, and can be
characterized as
\begin{equation}
\label{iowa-drury}
\mathbf H({\mathbb B}_N)=\left\{f(z)=\sum_{\alpha\in\mathbb N_0^N} z^\alpha f_\alpha\,\,:\,\,\|f\|^2_{\mathbf H({\mathbb B}_N)}=
\sum_{\alpha\in\mathbb N_0^N}\frac{\alpha!}{|\alpha|!}|f_\alpha|^2<\infty\right\}.
\end{equation}
For $N>1$ the Drury-Arveson space is contractively included in, but different from, the Hardy space of the ball. The latter has reproducing kernel
\[
\frac{1}{(1-\langle z,w\rangle)^N},\quad z,w\in\mathbb B_N.
\]
See \cite{akap3} for an expression for the inner product (not in terms of a surface integral).\\

We define $\left(\mathbf H({\mathbb B}_N)\right)^{n\times m}$ as in \eqref{iowa-drury}, but with now $f_\alpha,g_\alpha\in\mathbb C^{n\times m}$ and define for $f,g\in\left(\mathbf H({\mathbb B}_N)\right)^{n\times m}$, with $g(z)=\sum_{\alpha\in\mathbb N_0^N} z^\alpha g_\alpha$,
\begin{eqnarray}
[f,g]_{\left(\mathbf H({\mathbb B}_N)\right)^{n\times m}}&=&\sum_{\alpha\in\mathbb N_0^N} g_\alpha^* f_\alpha,\quad\text{and}\\
\langle f,g\rangle_{\left(\mathbf H({\mathbb B}_N)\right)^{n\times m}}
&=&{\rm Tr}\,[f,g]_{\left(\mathbf H({\mathbb B}_N)\right)^{n\times m}}.
\end{eqnarray}
In the sequel we will not write anymore explicitly the space in these forms. We will write sometimes $\mathbb C^N$ instead of $\mathbb
C^{1\times N}$.\\

For $a\in\mathbb B_N$ we will use the notations $e_a$ and $b_a$ for the normalized Cauchy kernel and the $\mathbb C^N$-valued Blaschke factor at the point $a$ respectively, that is:
\begin{equation}
\label{blasBN}
e_a(z)=\frac{\sqrt{1-\|a\|^2}}{1-\langle z,a\rangle}\quad \text{and}\quad b_a(z)
=\frac{(1-\|a\|^2)^{1/2}}{1-\langle z,a\rangle}(z-a)(I_N-a^*a)^{-1/2}.
\end{equation}

Let $w\in{\mathbb B}_N$. Then (see \cite{rudin-ball}; another more analytic and maybe easier proof can be found in \cite{akap1}):
\begin{equation}\label{rudin}
\frac{1-b_a(z)b_a(w)^*}{1-zw^*}=\frac{1-aa^*}{(1-za^*)(1-w^*a)}, \quad
z,w\in{\mathbb B}_N.
\end{equation}

Gleason's problem is solvable in the Drury-Arveson space and in the Hardy space;
see \cite{ak4}.

For $a=0$ and by setting $g_u(z,0)=g_u(z)$, a solution is  given by
\[
g_u(z,0)=\int_0^1\frac{\partial}{\partial z_u}f(tz)dt=\sum_{\alpha\in\mathbb N_0^N}\frac{\alpha_u}{|\alpha|}z^{\alpha-\e_u},
\]
where $\e_u$ is the $N$-index with all the other entries equal to $0$, but the $u$-th one equal to $1$, and with the understanding that
\[
\frac{\alpha_u}{|\alpha|}z^{\alpha-\e_u}=0
\]
if $\alpha_u=0$. We set $(R_uf)(z)=\int_0^1\frac{\partial}{\partial z_u}f(tz)dt$. We thus have
\[
f(z)-f(0)=\sum_{u=1}^Nz_u(R_uf)(z).
\]
When $N=1$, then $R_1$ reduces to the classical backward-shift operator which to $f$ associates the
function $\frac{f(z)-f(0)}{z}$ for $z\not=0$ and $f^\prime(0)$ for $z=0$.

\section{Interpolation in the Drury-Arveson space}
\setcounter{equation}{0}

This section is based on \cite{akap1} and reviews the tools necessary to develop the maximum selection principle and the convergence result in the next section. We provide the proofs for completeness.

 \begin{proposition}\label{charonne}
Let $0\not=c\in{\mathbb C}^{n\times 1}$, $a\in{\mathbb B}_N$, and let
$f\in{\mathbf H}({\mathbb B}_N)^{n\times 1}$. Then
$$c^*f(a)=0\iff f(z)=B(z)g(z),$$
where $B$ is given by
\begin{equation}\label{pascal}
B(z)=U\left(\begin{array}{cc}b_a(z)&0_{1\times (n-1)}\\
0_{(n-1)\times N}&I_{n-1}\end{array}
\right),\end{equation}
where $b_a(z)\in\mathbb C^{1\times N}$, $U\in\mathbb C^{n\times n}$ is a unitary matrix with the first column equal to $\frac{c}{c^*c}$, and $g$ is an arbitrary
element of ${\bf H}({\mathbb B}_N)^{(N+n-1)\times 1}$.
\end{proposition}

\begin{proof} We recall the proof of the proposition; see \cite[Proposition 4.5, p. 15]{akap1}. We note that
\[
c^*U=\begin{pmatrix}1&0&\cdots&0\\
                               0&0&\cdots&0\\
                                 & & &\\
0&0&\cdots&0\end{pmatrix},
\]
 and hence
\[
c^*B(a)=\begin{pmatrix}1&0&\cdots&0\\
                               0&0&\cdots&0\\
                                 & & &\\
0&0&\cdots&0\end{pmatrix}\begin{pmatrix}0_{1\times N}&0_{1\times (n-1)}\\
0_{(n-1)\times N}&I_{n-1}\end{pmatrix}=0_{n\times (N+(n-1))}.
 \]
and so every function of the form $Bg$ with $g\in{\mathbf H}({\mathbb B}_N)^{(N+n-1)\times 1}$ is a solution of the interpolation problem.
To prove the converse statement, we first remark that
\[
\frac{I_n-B(z)B(w)^*}{1-zw^*}=\frac{cc^*}{c^*c}\frac{1-aa^*}{(1-za^*)(1-w^*a)}.
\]
It follows that the one dimensional subspace $\mathcal H_1$ of ${\mathbf H}({\mathbb B}_N)^n$ spanned by the vector $\frac{\frac{c}{c^*c}}{1-za^*}$
has reproducing kernel $\frac{I_n-B(z)B(w)^*}{1-zw^*}$. Thus the decomposition of kernels
\[
\frac{I_n}{1-zw^*}=\frac{I_n-B(z)B(w)^*}{1-zw^*}+\frac{B(z)B(w)^*}{1-zw^*}
\]
leads to an orthogonal decomposition of the space ${\bf H}({\mathbb B}_N)^n$ as
\[
{\mathbf H}({\mathbb B}_N)^n=\mathcal H_1\oplus\mathcal H_1^{\perp},
\]
where $\mathcal H_1^\perp$ is the subspace of ${\mathbf H}({\mathbb B}_N)^n$ consisting of functions $g$ such that
$c^*g(a)=0$. Since the reproducing kernel of $\mathcal H_1^\perp$ is $\frac{B(z)B(w)^*}{1-zw^*}$ we have
\[
\mathcal H_1^\perp=\left\{Bg\,;\, g\in {\mathbf H}({\mathbb B}_N)^{(N+n-1)}\right\},
\]
with norm
\[
\|Bg\|_{{\mathbf H}({\mathbb B}_N)^n}=\inf_{g\in{\mathbf H}({\mathbb B}_N)^{(N+n-1)}}\|g\|_{{\mathbf H}({\mathbb B}_N)^{(N+n-1)}}.
\]
\end{proof}

We note that we do not write the dependence of $B$ on $a$ and $c$.

\begin{definition}
The $\mathbb C^{n\times (N+n-1)}$-valued function $B$ is an elementary Blaschke factor. A (possibly infinite) Blaschke product is a product of terms of the form \eqref{pascal} of compatible (growing) sizes.
\end{definition}

\begin{remark} {\rm Let $\mathcal B$ be a $\mathbb C^{n\times m}$-valued Blaschke product (or taking values operators from
$\mathbb C^n$ into $\ell_2$ if $m=\infty$). Then $\mathcal B$ is a Schur multiplier, meaning that the kernel
$\frac{I_n-\mathcal B(z)\mathcal B(w)^*}{1-\langle z,w\rangle}$ is positive definite in $\mathbb B_N$. When $N>1$, the family of Schur multipliers is strictly
included in the family of fucntions analytic and contractive in the unit ball. For the realization theory of Schur multipliers, see for
instance \cite{MR2394102,btv}.}
\end{remark}

More generally than \eqref{charonne} we have (see \cite[Theorem 5.2, p. 17]{akap1}):

\begin{theorem}\label{qwerty}
Given $a_1,\ldots ,a_M\in{\mathbb B}_N$ and vectors $c_1,\ldots , c_M
\in{\mathbb C}^{n\times 1}$ different from
$0_{n\times 1}$, a function $f\in{\bf H}({\mathbb B}_N)^{n\times 1}$ satisfies
$$c_j^*f(a_j)=0,\quad j=1,\ldots, M$$
if and only if it is of the form $f(z)=B(z)u(z)$,
where $B(z)$ is a rational
${\mathbb C}^{n\times (n+k(N-1))}$-valued function, for some integer $k\le M$,
taking coisometric values on the boundary of ${\mathbb B}_N$, and $u$ is an
arbitrary element in ${\mathbf H}({\mathbb B}_N)^{(n+k(N-1))\times 1}$.
\end{theorem}

\begin{proof} Indeed, starting with $j=1$ we have that $f=B_1g_1$, where $B_1$ is given by \eqref{pascal} with $a=a_1$ (and
an appropriately constructed matrix $U$) and $g_1\in{\bf H}({\mathbb B}_N)^{(N+n-1)\times 1}$. The interpolation condition $c_2^*f(a_2)=0$ becomes
\begin{equation}
\label{cond2}
c_2^*B_1(a_2)g_1(a_2)=0.
\end{equation}
If $c_2^*B_1(a_2)=0_{1\times (N+n-1)}$, any $g_1$ will be a solution. Otherwise, we solve \eqref{cond2} using Proposition \ref{charonne} and get
\[
g_1(z)=B_2(z)g_2(z),
\]
where $B_2$ is $\mathbb C^{(n+(N-1))\times (n+2(N-1))}$-valued and obtained from \eqref{pascal} with $a=a_2$ and an appropriately constructed matrix $U$.
Iterating this proceduce we obtain the result. The fact that $k$ may be strictly smaller than $M$ comes from the possibility that conditions as \eqref{cond2} occur. This will
not happen when $N=1$ and when all the $a_j$ chosen are different.
\end{proof}
\section{The maximum selection principle}
\setcounter{equation}{0}

The proof is similar to the one in the original paper \cite{QWa1} and in \cite{acqs}, but one relevant difference is the
use of orthogonal projections in $\mathbb C^{n\times n}$ of fixed rank. The fact that the set of such projections is compact in
$\mathbb C^{n\times n}$ ensures the existence of a maximum. Besides the use of the normalized Cauchy kernel,
the possibility of approximating by polynomials is a key tool in the proof.

\begin{proposition}
\label{prop42}
Let $B$ be a $\mathbb C^{u\times n}$-valued rational function of the variables $z_1,\ldots, z_N$, analytic in an neighborhood of the
closed unit ball $\overline{\mathbb B_N}$, and taking co-isometric values on the unit sphere, let
$r_0\in\left\{1,\ldots, n\right\}$, and let $F\in\mathbf H({\mathbb B}_N)^{n\times m}$. There exists $w_0\in\mathbb B_N$ and
a $\mathbb C^{n\times n}$-valued orthogonal projection $P_0$ of rank $r_0$ such that
\[
(1-\|w_0\|^2)\left({\rm Tr}~[B(w_0)P_0F(w_0),B(w_0)P_0F(w_0)]\right)\,\,\,\text{is maximum.}
\]
\end{proposition}

\begin{proof}
We first recall that for $f\in\mathbf H(\mathbb B_N)$ (that is, $n=m=1$), with power series $f(z)=\sum_{\alpha\in\mathbb N_0^N} f_\alpha z^\alpha$, and for
$w\in\mathbb B_N$, we have
\begin{equation}
\label{ineq11}
\sqrt{1-\|w\|^2}|f(w)|=|[f,e_{w}]|\le \|f\|.
\end{equation}
Let $F=(f_{ij})\in\mathbf H(\mathbb B_N)^{n\times m}$, where the entries $f_{ij}\in\mathbf H(\mathbb B_N)$ ($i=1,\ldots, n$ and $j=1,\ldots m$),
and let $P$ denote a projection of rank $r_0$. Then:
\[
\begin{split}
{\rm Tr}~ F(w)^*P B(w)^*B(w)PF(w)&\le {\rm Tr}~F(w)^*F(w)\quad (\text{\rm since $B(w)$ is contractive inside the sphere})\\
&=\sum_{i=1}^n\sum_{j=1}^m|f_{ij}(w)|^2.
\end{split}
\]
Hence, using \eqref{ineq11} for every $f_{ij}$, we obtain
\begin{equation}
\label{macau1234}
(1-\|w\|^2)\left({\rm Tr}~[B(w)P F(w),B(w)P F(w)]\right)\le\sum_{i=1}^n\sum_{j=1}^m\|f_{ij}\|^2=\|F\|^2.
\end{equation}
Let $\epsilon>0$. In view of the power series expansion characterization \eqref{iowa-drury} of the elements of the Drury--Arveson space,
there exists a $\mathbb C^{n\times m}$-valued polynomial $p$ in
$z_1,\ldots, z_N$ such that $\|F-p\|\le \epsilon$.
We have
\[
\begin{split}
(1-\|w\|^2)\left({\rm Tr}~[B(w)P F(w),B(w)PF(w)]\right)&\\
&\hspace{-3cm}
\le(1-\|w\|^2)\left({\rm Tr}~[F(w),F(w)]\right)\\
&\hspace{-3cm}=(1-\|w\|^2)\|(F-p)(w)+ p(w)\|^2\\
&\hspace{-3cm}\le2(1-\|w\|^2)\left(\|(F-p)(w)\|+\| p(w)\|\right)^2\\
&\hspace{-3cm}\le2(1-\|w\|^2)\|(F-p)(w)\|^2+2(1-\|w\|^2)\| p(w)\|^2\\
&\hspace{-3cm}\le 2\|F-p\|^2+2(1-\|w\|^2)\| p(w)\|^2 \quad (\text{where we have used \eqref{ineq11}})\\
&\hspace{-3cm}\le 2\epsilon^2+2(1-\|w\|^2)\|p(w)\|^2 .
\end{split}
\]
Since $(1-\|w\|^2)\|p(w)\|^2$ tends to $0$ as $w$ approaches the unit sphere,
the expression $(1-\|w\|^2)\left({\rm Tr}~[B(w)PF(w), B(w)PF(w)]\right)$ can be made arbitrary small, uniformly with
respect to $P$, as $w$ approaches the unit sphere. Thus,
\[
(1-\|w\|^2)\left({\rm Tr}~[B(w)P F(w), B(w)P F(w)]\right)
\]
is uniformly bounded as $w\in\mathbb B_N$ and $P$ runs through the projections of rank $r_0$, and goes to $0$ as $w$ tends to the boundary.
It has therefore a finite supremum, which is in fact a maximum and is in $\mathbb B_N$ (and not on the boundary), as
is seen by taking a subsequence tending to this supremum,  and this ends the proof.
\end{proof}

Let us rewrite $F(z)$ as
\begin{equation}
\label{bastille1}
F(z)=P_0F(w_0)e_{w_0}(z)\sqrt{1-\|w_0\|^2}+F(z)-P_0F(w_0)e_{w_0}(z)\sqrt{1-\|w_0\|^2}.
\end{equation}
We now show that \eqref{bastille1} gives an orthogonal decomposition of $F$, which is the first step in the
expansion of $F$ that we are looking for (see \eqref{bastille11} for a more precise way of writing the decomposition)
and for the algorithm that will arise repeating this construction.
\begin{lemma}
Let
\[
\begin{split}
H(z)&=F(z)-P_0F(w_0)e_{w_0}(z)\sqrt{1-\|w_0\|^2}\\
H_0(z)&=P_0F(w_0)e_{w_0}(z)\sqrt{1-\|w_0\|^2},
\end{split}
\]
where $w_0,P_0$ are as in Proposition \ref{prop42}. It holds that
\begin{equation}
\label{xiHz0}
P_0H({w_0})=0
\end{equation}
and
\[
[F,F]=[H_0,H_0]+[H,H].
\]
\end{lemma}
\begin{proof}
First we have \eqref{xiHz0} since
\[
P_0H(w_0)=P_0F(w_0)-P_0F(w_0)e_{w_0}(w_0)\sqrt{1-\|w_0\|^2}=0.
\]
Using \eqref{xiHz0} we have
\[
[H,P_0F(w_0)e_{w_0}(z)\sqrt{1-\|w_0\|^2}]=F(w_0)^*P_0H(w_0)(1-\|w_0\|^2)=0.
\]
So, $[H,H_0]=0$ and
\[
[F,F]=[H_0+H,H_0+H]=[H_0,H_0]+[H,H].
\]
\end{proof}

\section{The algorithm}
\label{sec3}
\setcounter{equation}{0}

To proceed and take care of the condition \eqref{xiHz0} (that is, in the scalar case, to divide by a Blaschke factor) we use a factor
of the form \eqref{pascal}. Then, we use Theorem \ref{qwerty} to find a $\mathbb C^{n\times (n+r_0^\prime (N-1))}$-valued rational
function $B_{w_0,P_0}$ with $r_0^\prime\le r_0$ and such that
\[
{\rm ran}\, P_0e_{w_0}=\mathbf H(\mathbb B_N)^{n\times m}\ominus B_{w_0,P_0} (\mathbf H(\mathbb B_N))^{(n+r_0^\prime (N-1))\times m},
\]
and so
\begin{equation}
\label{leiden}
\mathbf H(\mathbb B_N)^{n\times m}=
\left(\mathbf H(\mathbb B_N)^{n\times m}\ominus B_{w_0,P_0}(\mathbf H(\mathbb B_N))^{(n+r_0^\prime (N-1))\times m}\right)\oplus
B_{w_0,P_0}\mathbf H(\mathbb B_N))^{(n+r_0^\prime (N-1))\times m}.
\end{equation}

Let $F\in\left(\mathbf H(\mathbb B_N)\right)^{n\times m}$. We choose $w_0\in\mathbb B_N$ and $r_0\in\left\{1,\ldots, n\right\}$.
Using the maximum selection principle with $B(z)=I_n$ we get a decomposition of the form \eqref{leiden}. We rewrite \eqref{bastille1} as

\begin{equation}
\label{bastille11}
F(z)=P_0F(w_0)e_{w_0}(z)\sqrt{1-\|w_0\|^2}+B_{w_0,P_0}(z)F_1(z),
\end{equation}
where $F_1\in\left(\mathbf H(\mathbb B_N)\right)^{(n+r_0^\prime(N-1))\times m}$ (which, as $F_1$ is uniquely defined when $N>1$).
We now select $w_1\in\mathbb B_N$ and $r_1\in\left\{1,\ldots,n+r_0^\prime(N-1)\right\}$, and apply the maximum
selection principle to the pair $(B_{w_0,P_0}(z),F_1(z))$. We have then
\begin{equation}
F_1(z)=P_1F_1(w_1)e_{w_1}(z)\sqrt{1-\|w_1\|^2}+B_{w_1,P_1}(z)F_2(z),
\end{equation}
where $F_2\in\left(\mathbf H(\mathbb B_N)\right)^{(n+(r_0^\prime+r_1^\prime)(N-1))\times m}$ (with $r_1^\prime\le r_1$)
is not uniquely defined when $N>1$. So
\[
\begin{split}
F(z)&=P_0F(w_0)e_{w_0}(z)\sqrt{1-\|w_0\|^2}+B_{w_0,P_0}(z)P_1F_1(w_1)e_{w_1}(z)\sqrt{1-\|w_1\|^2}+\\
&\hspace{5mm}+B_{w_0,P_0}(z)B_{w_1,P_1}(z)F_2(z).
\end{split}
\]

We iterate the procedure with the pair $(B_{w_0,P_0}(z)B_{w_1,P_1}(z), F_2(z))$ and  observe the appearance of the Blaschke product
\[
\mathcal B_k(z)=B_{w_0,P_0}B_{w_1,P_1}B_{w_2,P_2}\cdots
B_{w_{k-1},P_{k-1}},\,\,\text{\rm for}\,\, k\ge 1,
\]
which will be $\mathbb C^{n\times (1+s_k(N-1))}$-valued for some $s_k\le\sum_{j=0}^{k-1}r_j$.
We set
\begin{equation}
\label{maastricht}
M_k=F_k(w_k)\in\mathbb C^{s_k\times m},
\end{equation}
and
\[
\mathfrak{B}_k(z)=\begin{cases}\,\, \sqrt{1-\|w_0\|^2}e_{w_0}(z)\,\,\quad\text{\rm for}\,\, k=0,\\
                                                                \,\, \sqrt{1-\|w_k\|^2}e_{w_k}(z)B_{w_0,P_0}(z)B_{w_1,P_1}(z)B_{w_2,P_2}(z)\cdots
B_{w_{k-1},P_{k-1}}(z)\,\,\text{\rm for}\,\, k\ge 1.\end{cases}
\]
Note that
\begin{equation}
\label{roosendaal}
\mathfrak B_k(w_k)=\mathcal B_k(w_k),\quad k\ge 1.
\end{equation}
We have
\begin{equation}
F(z)=\sum_{k=0}^{u}\mathfrak B_k(z)M_k+\mathcal B_{u+1}(z)F_{u+1}(z).
\label{cologne}
\end{equation}
Moreover,
\begin{equation}
\label{groningen}
\langle \mathfrak B_kM_k\, ,\,\mathfrak B_\ell M_\ell\rangle_{\mathbf H(\mathbb B_N)}=0\quad\text{\rm for}\quad k\not=\ell
\end{equation}
and we have by the orthogonality of the decomposition that
\begin{equation}
\label{groningen2}
\|F\|^2_{\mathbf H(\mathbb B_N)}=
\sum_{k=0}^u\|\mathfrak B_kM_k\|^2_{\mathbf H(\mathbb B_N)}+\|
\mathcal B_{u+1}(z)F_{u+1}\|^2_{\mathbf H(\mathbb B_N)}.
\end{equation}

This recursive procedure gives, at the $k$-th step, the best approximation. However we have to ensure that when $k$ tends to
infinity the algorithm converges. This is guaranteed by virtue of the next result.

\begin{theorem}
Suppose that in \eqref{cologne} at each step one selects $w_k$ and $P_k$ according to the maximum selection principle applied to
$(\mathcal B_{w_k}(z), F_k(z))$. Then the algorithm converges, meaning that
\[
F(z)=\sum_{k=0}^\infty \mathfrak{B}_k(z)M_k
\]
in the norm of the Drury-Arveson space.
\end{theorem}

\begin{proof}
We follow the arguments of \cite{QWa1} and \cite{acqs}.
We set
\begin{equation}
\label{Delft}
R_u(z)=F(z)-\sum_{k=0}^u \mathfrak{B}_k(z)M_k=\mathcal B_{w_{u+1}}(z)F_{u+1}(z)
\end{equation}
(where $F_{u+1}$ is not uniquely defined when $N>1$) and
\[
S_u(z)=\sum_{k=u+1}^{\infty}\mathfrak{B}_k(z)M_k.
\]
In view of \eqref{groningen}-\eqref{groningen2} the sum $\sum_{k=0}^\infty \mathfrak{B}_k(z)M_k$ converges
in the Drury-Arveson space.
Let $G$ be its limit, and assume that $G\not= F$. Thus there exists $w\in\mathbb B_N$ such that $G(w)\not=F(w)$.
We now proceed in a number of steps to obtain a contradiction.\\

STEP 1: {\sl There exists $u_0\in\mathbb N$ such that for $u\ge u_0$
\begin{equation}
\label{den_Haag}
\sqrt{1-\|w\|^2}\cdot\|R_u(w)\|>\sup_{\substack{c\in\mathbb C^n,\,\|c\|=1\\
d\in\mathbb C^m,\,\|d\|=1}}\frac{|\langle (F-G)d\,,\,c e_w\rangle_{(\mathbf H({\mathbb B_N))^n}}|}{2}.
\end{equation}
}
Indeed, $S_u$ tends to $0$ in norm in $\left(\mathbf H(\mathbb B_N)\right)^{n\times m}$.
Since in a reproducing kernel Hilbert space convergence in norm implies pointwise convergence,
we have $\lim_{u\rightarrow\infty} S_u(w)=0_{n\times m}$ in the norm of $\mathbb C^{n\times m}$, and there exists $u_0\in\mathbb N$ such that
\[
u\ge u_0\,\,\Longrightarrow\,\, \|S_u(w)\|<\frac{\|F(w)-G(w)\|}{2}.
\]
Thus
\[
\|R_u(w)\|+\frac{\|F(w)-G(w)\|}{2}>\|R_u(w)\|+\|S_u(w)\|\ge \|F(w)-G(w)\|,
\]
and so
\[
\|R_u(w)\|>\frac{\|F(w)-G(w)\|}{2},
\]
which can be rewritten as \eqref{den_Haag}.\\

STEP 2: {\sl It holds that
\begin{equation}
\label{amsterdam_zuid}
\lim_{k\rightarrow\infty} (1-\|w_k\|^2)\|\mathfrak B_k(w_k)M_k\|^2=0
\end{equation}
}

Indeed, from the convergence of $\sum_{k=0}^\infty\mathfrak{B}_kM_k$  we have
\[
\lim_{k\rightarrow\infty}\|\mathfrak{B}_kM_k\|_{({\mathbf H({\mathbb B_N)}})^{n\times m}}=0.
\]
Thus, with $c\in\mathbb C^m$ and $d\in\mathbb C^n$, we have:
\[
\begin{split}
|\langle \mathfrak B_k(w_k)M_kc,d\rangle_{({\mathbf H({\mathbb B_N)}})^n}
|&=|\langle \mathfrak B_kM_kc, \frac{d}{1-\langle\cdot, w_k\rangle}\rangle_{({\mathbf H}(\mathbb B_N))^n}|\\
&\le \|\mathfrak B_kM_kc\|_{({\mathbf H}(\mathbb B_N))^n}\cdot\frac{\|d\|}{\sqrt{1-\|w_k\|^2}}\\
&\le \|\mathfrak B_kM_k\|_{({\mathbf H}(\mathbb B_N))^{n\times m}}\cdot\|c\|\cdot\frac{\|d\|}{\sqrt{1-\|w_k\|^2}},
\end{split}
\]
where we have used the Cauchy-Schwarz inequality. So, after taking supremum on $c$ and $d$,
\[
\|\sqrt{1-\|w_k\|^2}\mathfrak B_k(w_k)M_k\|\le \|\mathfrak B_kM_k\|_{({\mathbf H({\mathbb B_N)})^{n\times m}}}\,\,\longrightarrow\,\, 0\,\,
\text{\rm as $n\rightarrow\infty$,}
\]
and so \eqref{amsterdam_zuid} holds in view of \eqref{roosendaal}.
\mbox{}\\

STEP 3: {\sl We conclude the proof.}\\

Let $u\ge u_0$, where $u_0$ is as in Step 1.
Since  $R_u(z)=\mathcal B_{w_{u+1}}(z)F_{u+1}(z)$ and since $w$ is such that $F(w)\not=G(w)$ we have

\begin{equation}
\label{beilen}
\sqrt{1-\|w\|^2}\cdot \|\mathcal B_{w_{u+1}}(w)F_{u+1}(w)\|>
\sup_{\substack{c\in\mathbb C^n,\,\|c\|=1\\
d\in\mathbb C^m,\,\|d\|=1}}\frac{|\langle (F-G)d\,,\,c e_w\rangle_{(\mathbf H({\mathbb B_N))^n}}|}{2}.
s\end{equation}
By definition of $w_{u+1}$ we have
\[
\sqrt{1-\|w_{u+1}\|^2}\cdot \|\mathcal B_{w_{u+1}}(w_{u+1})F_{u+1}(w_{u+1})\|<
\sup_{\substack{c\in\mathbb C^n,\,\|c\|=1\\
d\in\mathbb C^m,\,\|d\|=1}}\frac{|\langle (F-G)d\,,\,c e_w\rangle_{(\mathbf H({\mathbb B_N))^n}}|}{2},
\]
and using \eqref{roosendaal} we contradict \eqref{amsterdam_zuid}.

\end{proof}

\section{Infinite Blaschke products}
\setcounter{equation}{0}
In the previous sections appeared the counterpart of finite Blaschke products in the setting of the ball. We now consider
the case of infinite products.\smallskip

Let $a\in\mathbb B_N$, and let $b_a(z)$ be a $\mathbb C^{1\times N}$-valued Blaschke factor. We use the formula
\begin{equation}
b_a(z)=\frac{a-\dfrac{za^*}{aa^*}a-\sqrt{1-aa^*}\left(z-\dfrac{za^*}{aa^*}a\right)}{1-za^*},
\end{equation}
from \cite[(2), p. 25]{rudin-ball} rather than the formula in \eqref{blasBN}. See \cite[Lemma 4.2,p. 13]{akap1} for the equality between the two
expressions.\smallskip

We first prove a technical lemma useful in the proof of the convergence of an infinite Blaschke product.

\begin{lemma}
Let $\alpha=\frac{-a}{\sqrt{aa^*}}\in\partial\mathbb B_N$. Then,
\begin{equation}
\label{formula111}
\begin{split}
b_a(z)-b_a(\alpha)&=\frac{(z-\alpha)\left(a^*a\left(\dfrac{1-\sqrt{1-aa^*}}{aa^*}\right)-I_N\right)+z(\alpha a^*)-\alpha(za^*)
}{(1-za^*)(1+\sqrt{aa^*})}\cdot\sqrt{1-{aa^*}}\\
\end{split}
\end{equation}
and
\begin{equation}
\label{rehovot}
\|b_a(z)-b_a(\alpha)\|\le \frac{4\sqrt{1-aa^*}}{1-\|z\|}.
\end{equation}
\end{lemma}

\begin{proof}
We write $b_a(z)-b_a(\alpha)=\frac{\Delta}{(1-za^*)(1-\alpha a^*)}$, where the numerator
\[
\begin{split}
\Delta &=\left(a-\dfrac{za^*}{aa^*}a-\sqrt{1-aa^*}\left(z-\dfrac{za^*}{aa^*}a\right)\right)(1-\alpha a^*)-\\
&\hspace{5mm}-
\left(a-\dfrac{\alpha a^*}{aa^*}a-\sqrt{1-aa^*}\left(\alpha-\dfrac{\alpha a^*}{aa^*}a\right)\right)(1-za^*)
\end{split}
\]
has $16$ terms. Out of there, $a$ and $-a$ cancel each other, and
\[
\frac{za^*}{aa^*}a(\alpha a^*)=\frac{\alpha a^*}{aa^*}a(za^*)
\]
and
\[
\sqrt{1-aa^*}\frac{za^*}{aa^*}a(\alpha a^*)=\sqrt{1-aa^*}\frac{\alpha a^*}{aa^*}a(za^*).
\]
We are thus left with $10$ terms, which can be rewritten as:
\[
\Delta=(z-\alpha)\left(\left(-\frac{a^*a}{aa^*}-\sqrt{1-aa^*}I_N+\sqrt{1-aa^*}\frac{a^*a}{aa^*}
+a^*a\right)+\sqrt{1-aa^*}\left(z(\alpha a^*)-\alpha(za^*)\right)\right).
\]
Note that $z(\alpha a^*)-\alpha(za^*)$ does not vanish when $N>1$.
Therefore
\begin{equation}
\begin{split}
b_a(z)-b_a(\alpha)&=\\
&\hspace{-2cm}=\frac{(z-\alpha)\left(\left(-\frac{a^*a}{aa^*}-\sqrt{1-aa^*}I_N+\sqrt{1-aa^*}\frac{a^*a}{aa^*}
+a^*a\right)+\sqrt{1-aa^*}\left(z(\alpha a^*)-\alpha(za^*)\right)\right)}{(1-za^*)(1+\sqrt{aa^*})}.
\end{split}
\end{equation}

\end{proof}

\begin{remark}
{\rm
We note that
\[
\|a^*a\left(\dfrac{1-\sqrt{1-aa^*}}{aa^*}\right)-I_N\|=\sqrt{1-aa^*},
\]
as can be seen by computing the eigenvalues of the matrix in the left hand side.}
\end{remark}

We now consider a term of the form \eqref{pascal} and write (where $\alpha=-\frac{a}{\sqrt{aa^*}}$ and $W$ is a unitary matrix to be determined)
\begin{equation}
\label{baUW}
\begin{split}
B_a(z)&=B(z)W\\
&=\left(U\left(\begin{array}{cc}b_a(\alpha)&0_{1\times (n-1)}\\
0_{(n-1)\times N}&I_{n-1}\end{array}\right)+U\left(\begin{array}{cc}(b_a(z)-b_a(\alpha)&0_{1\times (n-1)}\\
0_{(n-1)\times N}&I_{n-1}\end{array}\right)\right)W,
\end{split}
\end{equation}
where we do not stress the dependence on the matrices $U$ and $W$.
Since $b_a(\alpha)$ is a unit vector, the matrix
\[
U\left(\begin{array}{cc}b_a(\alpha)&0_{1\times (n-1)}\\
0_{(n-1)\times N}&I_{n-1}\end{array}\right)
\]
is coisometric, and we can complete the columns of its adjoint to a unitary matrix $W$. Then we have
\begin{equation}
\label{norm123}
U\left(\begin{array}{cc}b_a(\alpha)&0_{1\times (n-1)}\\
0_{(n-1)\times N}&I_{n-1}\end{array}\right)
W=\begin{pmatrix} I&0\end{pmatrix}
\end{equation}
and show that the corresponding infinite product will converge when $\sum_{n=1}^\infty\sqrt{1-a_na_n^*}$ converges.\\

In Theorem \ref{blaschkerectangle} below we imbed $\mathbb C^m$ inside $\ell_2$ via the formula:
\begin{equation}
\label{chapman}
i_m(z_1,\ldots, z_m)=(z_1,\ldots, z_m,0,0,\ldots).
\end{equation}
We also need some notation and introduce the matrices
\[
E_k=\begin{pmatrix}1&0_{1\times k(N-1)}\end{pmatrix}\quad (=1\,\,\text{when}\,\, N=1),
\]
\[
F_k=\begin{pmatrix}I_{1+(k-1)(N-1)}&0_{(1+(k-1)(N-1))\times(N-1)}\end{pmatrix}\in\mathbb C^{(1+(k-1)(N-1))\times(N+(k-1)(N-1))},
\]
and note that $E_1=F_1$ and
\begin{equation}
\label{010116!!!}
E_k=E_1F_2\cdots F_{k}\quad\text{and}\quad E_{k+1}=E_kF_{k+1}.
\end{equation}

We also note that multiplication by $F_k$ on the right imbeds $\mathbb C^{1+(k-1)(N-1)}$ into $\mathbb C^{1+k(N-1)}$. It will be
useful to use the notation
\begin{equation}
F_{m_1}^{m_2}=\stackrel{\curvearrowright}{\prod_{k=m_1+1}^{m_2}}E_{k}.
\label{lalala}
\end{equation}
\begin{theorem}
\label{blaschkerectangle}
The infinite product $b_{w_0}(z)B_{w_1}(z)B_{w_2}(z)\cdots B_{w_{k-1}}(z)\cdots$ where the factors are normalized as in \eqref{norm123}
converges pointwise for $z\in\mathbb B_N$ to a non-identically vanishing $\ell_2$-valued function analytic in $\mathbb B_N$ if
\begin{equation}
\sum_{k=0}^\infty \sqrt{1-a_ka_k^*}<\infty.
\label{cond}
\end{equation}
\end{theorem}
\begin{proof}
The idea is to follow the proof for the scalar case appearing in sources such as \cite{choquet,favard} and reproduced in
\cite[pp. 104-105]{CAPB}.
We consider the product
\[
\stackrel{\curvearrowright}{\prod_{k=1}^m}(F_k+A_k(z))
\]
with
\[
A_k(z)=U_k\left(\begin{array}{cc}b_{a_k}(z)-b_{a_k}(\alpha)&0_{1\times (k-1)(N-1)}\\
0_{(k-1)(N-1)\times N}&I_{(k-1)(N-1)}\end{array}\right)W_k\,\,\in\,\,\mathbb C^{(1+(k-1)(N-1))\times(N+(k-1)(N-1))},
\]

and note that, in view of \eqref{rehovot},
\begin{equation}
\label{rehovot2}
\|A_k(z)\|\le\frac{4\sqrt{1-a_ka_k^*}}{1-\|z\|}.
\end{equation}
Following the classical proof we now prove the convergence in a number of steps and use
\cite[pp. 104-105]{CAPB} as a source.\\

Note that, to ease the notation, in Steps 1-3 we do not stress the dependence of $A_k$ on the variable $z$.\\

STEP 1: {\sl It holds that}
\begin{equation}
\label{janvier2016}
\|\stackrel{\curvearrowright}{\prod_{k=1}^m}(F_k+A_k)-E_{m}\|\le \prod_{k=1}^m(1+\|A_k\|)-1,\quad m\in\mathbb N.
\end{equation}
We proceed by induction, the case $m=1$ being trivial since $E_1=F_1$. We have
\begin{align*}
\|\stackrel{\curvearrowright}{\prod_{k=1}^{m+1}}(F_k+A_k)-E_{m+1}\|&=
\|\left(\prod_{k=1}^{m}(F_k+A_k)\right)(F_{m+1}+A_{m+1})-E_{m+1}\|\\
&=\|\left(\stackrel{\curvearrowright}{\prod_{k=1}^{m}}(F_k+A_k)\right)(F_{m+1}+A_{m+1})-E_{m}F_{m+1}\|\\
&\le \|\left(\left(\stackrel{\curvearrowright}{\prod_{k=1}^{m}}(F_k+A_k)\right)-E_m\right)F_{m+1})\|\,+\\
&\hspace{5mm}+
\|A_{m+1}\|\left(\prod_{k=1}^m(1+\|A_k\|)\right)\\ &\le
\left(\left(\prod_{k=1}^n(1+\|A_k\|)\right)-1\right)+\|A_{m+1}\|\left(\prod_{k=1}^m(1+\|A_k\|)\right)\\
&=\left(\prod_{k=1}^{m+1}(1+\|A_k\|)\right)-1,
\end{align*}
where we have used the induction hypothesis to go from the third to the fourth line.\medskip

Replacing $A_k$ by $A_{k+m_1}$  we have for $m_2>m_1$:
\begin{equation}
\label{Sentier}
\|\left(\stackrel{\curvearrowright}{\prod_{k=m_1+1}^{m_2}}(E_k+A_k)\right)-\stackrel{\curvearrowright}{\prod_{k=m_1+1}^{m_2}}E_{m_2}\|\le
\left(\prod_{k=m_1+1}^{m_2}(1+\|A_k\|)\right)-1.
\end{equation}

STEP 2: {\sl Let $Z_m=\stackrel{\curvearrowright}{\prod}_{k=1}^m(F_k+A_k)$. Then,}
\[
\|Z_m\|\le e^{\sum_{k=1}^m\|A_k\|}<\infty
\]

Indeed,
\[
\begin{split}
\|Z_m\|&\le\prod_{k=1}^m\|F_k+A_k\|\\
&\le\prod_{k=1}^m (1+\|A_k\|)\\
&\le \prod_{k=1}^m e^{\|A_k\|} \le e^{\sum_{k=1}^\infty
\|A_k\|}<\infty,
\end{split}
\]
in view of \eqref{cond} and \eqref{rehovot2}.\\

STEP 3: {\sl Let $i_m$ be defined by \eqref{chapman}. Then,
$(i_m(Z_m))_{m\in\mathbb N}$ is a Cauchy sequence in $\ell_2$}.\\

For $m_2>m_1$  and using \eqref{Sentier}, we have
\begin{equation}
\label{eq:b1}
\begin{split}
\|i_{m_2}(Z_{m_2})-i_{m_1}(Z_{m_1})\|_{\ell_2}&=\|Z_{m_2}-i_{m_2}\cdots i_{m_1+1}(Z_{m_1})\|_{\mathbb C^{1\times
(1+(m_2+1)(N-1))}}\\
&=\left(\stackrel{\curvearrowright}{\prod_{k=1}^{m_1}}(F_k+A_k)\right)\cdot\left(\stackrel{\curvearrowright}{\prod_{k=m_1+1}^{m_2}}(F_k+A_k)-
F_{m_1+1}^{m_2}
\right)\\
&\le\left(\prod_{k=1}^{m_1}(1+\|A_k\|)\right)\cdot\|
\stackrel{\curvearrowright}{\prod_{k=m_1+1}^{m_2}}(F_k+A_k)-F^{m_2}_{m_1+1}\|\\
&\le e^K\left\{\left(\prod_{k={m_1}+1}^{m_2}(1+\|A_k\|)\right)-1\right\}\\
&\le e^K\left\{\left(\prod_{k={m_1}+1}^{m_2}e^{\|A_k\|}\right)-1\right\}\\
&\le \left(\sum_{k=m_1+1}^{m_2}\|A_k\|\right)e^{2K},
\end{split}
\end{equation}
with $K=\sum_{k=1}^\infty\|A_k\|$ (which is finite, thanks to \eqref{cond} and \eqref{rehovot2}), and using inequality
\[
e^x\le1+xe^x,\quad x\ge 0,
\]
with $x=\sum_{k=m_1+1}^{m_2}\|A_k\|$.\\

STEP 4: {\sl The $\ell_2$-valued function  $Z(z)=\lim_{m\rightarrow\infty} Z_m(z)$ does not vanish identically in $\mathbb B_N$.}\smallskip

We first assume that $\sum_{k=1}^\infty \|A_k(z)\|<\frac{1}{2}$ and prove by induction that
\begin{equation}
\label{bilbao}
\|Z_m(z)\|\ge 1-\sum_{k=1}^m\|A_k(z)\|.
\end{equation}
The claim $Z\not=0$ will then follow by letting $m\rightarrow\infty$. For $m=1$ the claim is trivial. Assume that \eqref{bilbao} holds for
$m$. We then have:
\[
\begin{split}
\|Z_{m+1}(z)\|&=\|Z_m(z)(F_{m+1}+A_{m+1}(z))\|\\
&\ge\|Z_m(z)F_{m+1}\|-\|Z_m(z) A_{m+1}(z)\|\quad\text{(since $\|Z_m(z)F_{m+1}\|=\|Z_m(z)\|$)}\\
&\ge\|Z_m(z)\|-\|Z_m(z)\|\|A_{m+1}(z)\|\quad\hspace{.5cm}\text{(since $\|Z_m(z)A_{m+1}(z)\|\le\|Z_m(z)\|\|A_{m+1}(z)\|$)}\\
&=\|Z_m(z)\|\cdot(1-\|A_{m+1}(z)\|)\\
&\ge( 1-\sum_{k=1}^m\|A_k(z)\|)(1-\|A_{m+1}(z)\|)\\
&\ge( 1-\sum_{k=1}^{m+1}\|A_k(z)\|).
\end{split}
\]

Let $M\in\mathbb N$ (depending on $z$) be such that  $\sum_{k=M}^\infty \|A_k(z)\|<\frac{1}{2}$. Then the same inequality holds in an open neighborhood $V$ of $z$ in view of \eqref{rehovot2}, and so the same $M$ can be taken for $z\in V$.
Let
\[
Z_{M-1}(z)=\stackrel{\curvearrowright}{\prod}_{u=1}^{M-1}(F_u+A_u(z))\in\mathbb C^{1\times(1+(M-2)(N-1))},
\]
where
\[
\widetilde{Z_M}(z)=\stackrel{\curvearrowright}{\prod}_{u=M}^\infty (F_u+A_u(z)).
\]
We can patch together all the $Z_{M-1}(z)\widetilde{Z_M}(z)$ to a common function defined in $\mathbb B_N$.
Assume that $Z_{M-1}(z)\widetilde{Z_M}(z)\equiv 0$ in one of the neighborhoods $V$. Then the infinite product vanishes identically
in $\mathbb B_N$. Letting $z$ go to the boundary we get a contradiction since $Z_{M-1}(z)\widetilde{Z_M}(z)$ takes coisometric values on $\partial\mathbb B_N$.\\

STEP 5: {\sl Using \eqref{janvier2016} and \eqref{eq:b1}, we obtain the bound:}

\begin{equation}
\label{bound010116!!!}
\|\stackrel{\curvearrowright}{\prod_{k=1}^m}(F_k+A_k(z))-Z\|\le
e^{2K}\left(\sum_{k=m+1}^\infty \|A_k(z)\|\right).
\end{equation}
\end{proof}

It is worthwhile to note that the above theorem allows to further extending the results of \cite{akap1} to the case of an infinite number
of points.

\end{document}